\documentclass[10pt,a4paper]{amsart}
\usepackage[top=4.2cm,bottom=4.2cm,left=4.5cm,right=4.5cm]{geometry}

\usepackage{amsmath,amscd,amssymb,amsthm,setspace}
\usepackage[english]{babel}
\usepackage{mathtools}
\usepackage{enumerate}
\usepackage[numbers]{natbib}
\usepackage[hidelinks]{hyperref}
\usepackage{tikz}
\usepackage{tkz-graph}

\usepackage{xcolor}

\DeclareMathOperator{\IRR}{Irr}

\newcommand{\vF}{\mathbb{F}}
\newcommand{\vN}{\mathbb{N}}

\newcommand{\cS}{\mathcal{S}}

\newtheorem{theorem}{Theorem}[section]
\newtheorem{lemma}[theorem]{Lemma}
\newtheorem{proposition}[theorem]{Proposition}
\newtheorem{corollary}[theorem]{Corollary}
\theoremstyle{definition}

\newtheorem{definition}[theorem]{Definition}
\theoremstyle{remark}
\newtheorem{remark}[theorem]{Remark}
\newtheorem{example}[theorem]{Example}

\numberwithin{equation}{section}
\begin{document}

\title[On sets of irreducible polynomials closed by composition]
{On sets of irreducible polynomials closed by composition}

\author[A. Ferraguti]{Andrea Ferraguti}
\address{Institute of Mathematics\\
University of Zurich\\
Winterthurerstrasse 190\\
8057 Zurich, Switzerland\\
}
\email{andrea.ferraguti@math.uzh.ch}

\author[G. Micheli]{Giacomo Micheli}
\address{Mathematical Institute\\
University of Oxford\\
Woodstock Rd,\\ 
Oxford OX2 6GG, United Kingdom
}
\email{giacomo.micheli@maths.ox.ac.uk}

\author[R. Schnyder]{Reto Schnyder}
\address{Institute of Mathematics\\
University of Zurich\\
Winterthurerstrasse 190\\
8057 Zurich, Switzerland\\
}
\email{reto.schnyder@math.uzh.ch}

   \thanks{The Second Author is thankful to SNSF grant number 161757.}


\begin{abstract}
Let $\mathcal S$ be a set of monic degree $2$ polynomials over a
finite field and let $C$ be the compositional semigroup generated by $\mathcal
S$. In this paper we establish a necessary and sufficient condition for $C$ to
be consisting entirely of irreducible polynomials. The condition we
deduce depends on the finite data encoded in a certain graph uniquely
determined by the generating set $\cS$. Using this machinery we are able both
to show examples of semigroups of irreducible polynomials generated by two
degree $2$ polynomials and to give some non-existence results for some of these
sets in infinitely many prime fields satisfying certain arithmetic conditions.
\end{abstract}

\maketitle

\section{Introduction}
\label{sec:introduction}
Since irreducible polynomials play a fundamental role in applications and in the
whole theory of finite fields (see for
example~\cite{lidl1997finite,mullen2013handbook,rabin1981fingerprinting,barbulescu2014heuristic,ostafe2010length,ahmadi2012stable}),
related questions have a long history (see for example
\cite{gao1999irreducible,gao1997tests,shoup1990new,jones2012iterative,jones2012settled,andrade2013special,von2003irreducible}).
In this paper we specialize on irreducibility questions regarding
compositional semigroups of polynomials. This kind of question has been
addressed in the specific case of semigroups generated by a single quadratic
polynomial, see for example in
\cite{ahmadi2012stable,ostafe2010length,jones2012settled,jones2012iterative},
for analogous results related to additive polynomials, see
\cite{batra1994algebraic,batra1994}. It is worth mentioning that one of these
results \cite[Lemma 2.5]{jones2012settled} has been
recently used in \cite{ferraguti2016existence} by the first and the second
author of the present paper to prove \cite[Conjecture 1.2]{andrade2013special}. 

Throughout the paper, $q$ will be an odd prime power, $\vF_q[x]$ the univariate
polynomial ring over the finite field $\vF_q$ and $\IRR(\vF_q[x])$ the set of
irreducible polynomials in $\vF_q[x]$.
Let us give an example which motivates this paper.
For a prime $p$ congruent to $1$ modulo $4$, we can fix in $\vF_{p}[x]$ two
quadratic polynomials $f = (x-a)^2 + a$ and $g = (x-a-1)^2 + a$ such
that both $a$
and $ a + 1$ are non-squares in $\vF_p$. One can experimentally check that any
possible composition of a sequence of $f$'s and $g$'s is irreducible (for a
concrete example, take $q=13$, $(x-5)^2 + 5$ and $g = (x-6)^2 + 5$).
Let us denote the set of such compositions by $C$.
A couple of observations are now necessary:
\begin{itemize} 
\item In principle, it is unclear whether a finite number of irreducibility checks will ensure that $C$ is a subset of $\IRR(\vF_q[x])$.
\item The fact that $C\subseteq \IRR(\vF_{q}[x])$ is indeed pretty unlikely to happen by chance, as the density of degree $2^n$ monic irreducible polynomials over $\vF_q$ is roughly $1/2^n$. Thus, if $C$ satisfies this property, one reasonably expects that  there must be an algebraic reason for that.
\end{itemize}
We address these issues by giving a necessary and sufficient condition for the
semigroup $C\subset \vF_q[x]$ to be contained in $\IRR(\vF_q[x])$. In addition,
this condition is algebraic and can be checked by performing only a finite
amount of computation over $\vF_q$, answering both points above.

In Section \ref{sec:gen_crit} we describe the criterion (Theorem \ref{thm:main} and Corollary \ref{main_corollary}) and provide a non-trivial example (Example \ref{ex:irred_semigroup}) of a compositional semigroup in $\vF_q[x]$ contained in $\IRR(\vF_q[x])$ and generated by two polynomials.

In Section \ref{sec:p3mod4} we show the non-existence of such $C$ whenever $q$
is a prime congruent to $3$ modulo $4$ and the generating polynomials are of a
certain form (Proposition \ref{prop:p3}). Example \ref{ex:sharp_cond} shows
that these conditions are indeed sharp.

\section{A general criterion}\label{sec:gen_crit}

In order to state our main result, we first need the following definition, which describes how to build a finite graph encoding only the useful (to our purposes) information contained in the generating set of the semigroup.

\begin{definition}\label{graph_definition}
Let $q$ be an odd prime power, $\vF_q$ the finite field of order $q$ and $\cS$ a subset of $\vF_q[x]$.
We denote by $G_\cS$ the directed multigraph defined as follows:
\begin{itemize}
\item the set of nodes of $G_\cS$ is $\vF_q$;
\item for any node $a \in \vF_q$ and any polynomial $f \in \cS$, there is a
	directed edge $a\rightarrow f(a)$. We label that edge with $f$.
\end{itemize}
\end{definition}

Before stating the next definition, we recall that for any monic polynomial $f$ 
of degree $2$ there exist unique pair $(a_f,b_f)\in \vF_{q}^2$ such that
$f=(x-a_f)^2-b_f$.

\begin{definition}
Let $\cS$ be a subset of $\vF_q[x]$ consisting of monic polynomials of degree $2$.
We call the set $D_\cS:=\{-b_f\,|\,f\in \cS \}\subseteq \vF_q$,  the $\cS$\emph{-distinguished set} of $\vF_q$.

\end{definition}
The following result is just an inductive extension of the classical Capelli's Lemma.
\begin{lemma}[Recursive Capelli's Lemma]
Let $K$ be a field and $f_1,\dots,f_l$ be a set of irreducible polynomials in
$K[X]$. The polynomial $f_1(f_2(\cdots (f_{l})\cdots))$ is irreducible if and only if the following conditions are satisfied
\[\begin{cases}
f_1 \mbox{is irreducible over $K[X]$ }\\ 
f_2-\alpha_1 \quad \mbox{is irreducible over $K(\alpha_1)[X]$ for a root $\alpha_1$ of $f_1$}\\
f_3-\alpha_2 \quad \mbox{is irreducible over $K(\alpha_1,\alpha_2)[X]$ for a root $\alpha_2$ of $f_2- \alpha_1$}\\
\cdots\\
f_{l}-\alpha_{l-1} \quad \mbox{is irreducible over $K(\alpha_1,\dots, \alpha_{l-1})[X]$ for a root $\alpha_{l-1}$ of $f_{l-1} - \alpha_{l-2}$}
\end{cases}
\]
\end{lemma}
\begin{proof}
Given Capelli's Lemma \cite[Lemma 2.4]{jones2012settled}, the proof is straightforward by induction.
\end{proof}
We are now ready to state and prove the main theorem.
\begin{theorem}\label{thm:main}
Let $\cS$ be a set of generators for a compositional semigroup $C\subseteq \vF_q[x]$.
Suppose that $\cS$ consists of polynomials of degree $2$.
Then we have that $C\subseteq \IRR(\vF_q[x])$ if and only if no element of
$-D_{\cS}=\{b_f\,|\,f\in \cS \}\subseteq \vF_q$ is a square and
 in $G_\cS$  there is no path of positive length from a node of $D_\cS$ to a square of $\vF_q$.
\end{theorem}
\begin{proof}
It is clear that $C$ contains a reducible polynomial of degree $2$ if and only if one element of $-D_\cS$ is a square.
Thus we can assume that $\cS$ consists only of irreducible polynomials.

We now show that in $G_\cS$ there is a path of positive length from a node of $D_{\cS}$ to a square if and only if $C$ contains a reducible polynomial of degree greater or equal than $4$.

First, suppose that the composition $f_1 f_2\cdots f_{l+1}$ is a reducible polynomial of minimal degree, with $f_i\in \cS$ and $f_i=(x-a_i)^2-b_i$, for $i\in\{1,\dots, l+1\}$ and $l\geq 1$.
Whenever $\beta$ is not a square in $\vF_q$, we denote by $\sqrt{\beta}$ a root of the polynomial $T^2-\beta$ in the algebraic closure of $\vF_q$.
By Capelli's Lemma applied to the composition of $f_1\cdots f_l$ and by the minimality of the degree of $f_1 f_2\cdots f_{l+1}$, we have that the following elements are not squares in their field of definition:
\[\beta_0:=b_1\in \vF_q,\]
\[\beta_1:=b_2+a_1+\sqrt{\beta_0}\in \vF_{q^2}\]
\[\beta_2=b_3+a_2+\sqrt{\beta_1}\in \vF_{q^{2^2}}\]
\[\dots\]
\[\beta_{l-1}:=b_l+a_{l-1}+\sqrt{\beta_{l-2}}\in \vF_{q^{2^{l-1}}}.\]

On the other hand, $\beta_l=b_{l+1}+a_{l}+\sqrt{\beta_{l-1}} \in \vF_{q^{2^{l}}}$ is necessarily a square.  
For $j<i$, let us denote by $N_i^j:\vF_{q^{2^i}} \longrightarrow \vF_{q^{2^j}}$ the usual norm map.
We claim that the $\vF_q$-norm $N_l^0:\vF_{q^{2^l}}\longrightarrow \vF_q$ maps $\beta_l$ to $f_1 (\cdots f_{l}(-b_{l+1})\cdots)$, and this defines a path in $G_\cS$ from $-b_l$ to a square. This can be easily seen by first decomposing $N_l^1$:
\[N_l^1=N_2^{1} \circ N_{3}^{2} \circ \dots N_l^{l-1}\]
and then by directly computing $N_2^{1} \circ N_{3}^{2} \circ \dots N_l^{l-1}(\beta_l)$.
It is important indeed  that $\beta_0,\beta_1,\dots, \beta_{l-1}$ are not squares, as the computation above only gives the desired result when $(\sqrt{\beta_i}) ^{q^{2^i}}=-\sqrt{\beta_i}$.

Conversely, suppose that in $G_\cS$ there is a path to a square $s$. Choose such a path of minimal length, starting at some $-b_f$ in the distinguished set, for some $f\in \cS$. Consider now the composition associated to this path: 
if \[s=f_1 f_2 \cdots f_l(-b_f),\] set $f_{l+1} = f$ and let $g\coloneqq f_1 f_2 \cdots f_{l+1}\in \vF_q[x]$.
One can construct the $\beta_i$'s as before, i.e.\@ $\beta_0 = b_1$ and for $i\in \{1,\dots, l\}$, $\beta_i= b_{i+1} + a_i+\sqrt{\beta_{i-1}}$. We can suppose that the $\beta_i$'s for $i < l$ are all non-squares as otherwise, by taking the smallest $d$ such that $\beta_d$ is square, we find a composition $f_1 f_2 \cdots f_{d+1}$ that is reducible by Recursive Capelli's Lemma, and then we are done.

As all the $\beta_i$'s, for $i<l$, can be supposed to be non-squares, we have as above that $N_l^0(\beta_l)=f_1 f_2 \cdots f_l(-b_{l+1})=s$, which we have assumed to be a square. Now, recall that an element  of a finite field is a square if and only if its norm is a square: this  shows that $g$ is reducible by Recursive Capelli's Lemma.
\end{proof}

The reader should observe that this theorem generalizes \cite[Proposition 2.3]{jones2012settled}.
It is useful to mention the following corollary, which is immediate.
\begin{corollary}\label{main_corollary}
Let $\cS$ be a set of irreducible degree two polynomials and $C$ defined as in Theorem \ref{thm:main}. Then $C\subseteq \IRR(\vF_q[x])$ if and only if there is no path of positive length from a node of $D_\cS$ to a square of $\vF_q$.
\end{corollary}
\begin{proof}
It is enough to observe that whenever $\cS\subseteq \IRR(\vF_q[x])$ then $-D_\cS$ consists of non-squares.
\end{proof}

\begin{remark}
Given that $C$ is generated by degree $2$ polynomials, it is easy to observe that the datum of $\cS$ is equivalent to the datum of $C$.
\end{remark}
The following example shows a way to find examples of semigroups contained in $\IRR(\vF_q[x])$ when $q\equiv 1 \mod 4$.
\begin{example}\label{ex:irred_semigroup}
Let $q \equiv 1 \mod 4$ be a prime power, and let $a \in \vF_q$ such that both $a$
and $b = a + 1$ are non-squares.
Define $f = (x-a)^2 + a$ and $g = (x-b)^2 + a$.
In this situation, we have $D_\cS = \{ a \}$, and by assumption, $-a$,
$a$ and $b$ are all non-squares.
Since $f(a) = g(b) = a$ and $f(b) = g(a) = b$, all paths in $G_\cS$ starting
from $a$ end in a non-square, and the conditions of Theorem~\ref{thm:main} are
satisfied.
Figure~\ref{fig:irred_semigroup} shows the relevant part of the graph $G_\cS$.
The reader should observe that this is indeed the example mentioned in the introduction.

\begin{figure}[h]
	\centering
	\begin{tikzpicture}
		\node (a) at (-1,0) [draw, circle, double] {$a$};
		\node (b) at ( 1,0) [draw, circle]         {$b$};

		\draw [->] (a) to [out=45,in=135]  node[below] {$g$} (b);
		\draw [->] (b) to [out=225,in=315] node[above] {$g$} (a);
		\draw [->] (a) to [loop left]      node[left]  {$f$} (a);
		\draw [->] (b) to [loop right]     node[right] {$f$} (b);
	\end{tikzpicture}
	\caption{The nodes of $G_\cS$ reachable from $D_\cS$.}
	\label{fig:irred_semigroup}
\end{figure}
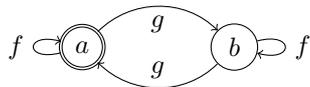
\end{example}

\section{The case \texorpdfstring{$p\equiv 3 \mod 4$}{}}\label{sec:p3mod4}
Whenever $q=p$ is a prime congruent to $3$ modulo $4$, we have the following non-existence results.
\begin{lemma}
 Let $p\equiv -1 \bmod 8$ be a prime, and let $f=x^2-b$ be a polynomial in $\vF_p[x]$. Let $C$ be the semigroup generated by $f$. Then $C$ contains a reducible polynomial.
\end{lemma}
\begin{proof}
 Assume for contradiction that $C \subset \IRR(\vF_p[x])$.
 First note that if $b$ is a square, then $f$ is reducible, so we can assume
 that $b$ is not a square, and thus $-b$ is a square. Consider the set of
 iterates $T=\{f(-b),f^2(-b),\ldots\}\subseteq \vF_p$. By Corollary
 \ref{main_corollary}, $C$ contains only irreducible polynomials if and only if
 $T$ contains only nonsquares. So assume that this condition holds. Since $T$ is
 finite, there exist $k < m\in \vN_{>0}$ such that $f^m(-b)=f^k(-b)$. Choose $k$
 to be minimal.
 Now there are two cases: if $k > 1$, then there exist two distinct
 elements $u,v\in T$ such that $u^2-b=v^2-b$. Thus, $u=-v$, which implies that
 one between $u$ and $v$ is a square, a contradiction.
 If on the other hand $k = 1$, then we have $f^m(-b) = f(-b) = b^2 - b$, and so
 $f^{m-1}(-b)$ is either $-b$ or $b$. It can't be $-b$, since that is a square,
 so we must have $f^{m-1}(-b) = b \in T$. Setting $u = f^{m-2}(-b)$, we get that
 $u^2 - b = b$ and so $u^2 = 2b$, which is a contradiction because $2$ is a
 square in $\vF_p$ and consequently $2b$ is not.
\end{proof}

\begin{proposition}\label{prop:p3}
 Let $p\equiv 3 \bmod 4$ be a prime. Let $f=x^2-b_f$ and $g=x^2-b_g$ be polynomials in $\vF_p[x]$ with $b_f,b_g$ distinct non-squares. Let $\mathcal S=\{f,g\}$ and let $C$ be the semigroup generated by $\mathcal S$. Then $C$ contains a reducible polynomial.
\end{proposition}
\begin{proof}
 Let $G_{\cS}$ be the graph attached to $\cS$ as in Definition
 \ref{graph_definition}. Let $G'_{\cS}$ be the induced subgraph consisting of
 all nodes of $G_{\cS}$ that are reachable by some path of
 positive length starting from $-b_f$ or $-b_g$.
 That is, the edges of $G'_{\cS}$ are just the edges of $G_{\cS}$ starting and
 ending at a node in $G'_{\cS}$.
 From now on, when we
 speak of nodes and edges, we will always be referring to nodes and edges in
 $G'_{\cS}$.
 We call an edge from $u$
 to $v$ an \emph{$f$-edge} if it comes from the relation $f(u)=v$, while we
 call it a \emph{$g$-edge} if it comes from $g(u)=v$.
 Since $b_f$ and $b_g$ are assumed nonsquare, we have by Corollary
 \ref{main_corollary} that $C$ contains a reducible polynomial if and only if at
 least one of the nodes of $G'_{\cS}$ is a square.
 In the following, we assume for contradiction that $G'_{\cS}$ consists only of
 non-squares.

 Let us observe the following: suppose that there exists a node $v$ of
 $G'_{\cS}$ which is the target of two $f$-edges. By definition, this means that
 there exist two distinct nodes $u,u'\in G'_{\cS}$ such that
 $u^2-b_f=u'^2-b_f=v$. This implies that $u'=-u$, and thus one between $u$ and
 $u'$ is a square, since $-1$ is not a square in $\vF_p$. This contradicts our
 assumption. 
 By symmetry, the same applies to $g$-edges.
 
 By the argument above, we see that every node is the target of at most one
 $f$-edge and one $g$-edge, and by counting edges that it is indeed
 exactly one of each.
 
 Now, consider the sum
 \[
	 \sum_{v \in G'_{\cS}} (f(v) - g(v)).
 \]
 On one hand, each node $u \in G'_{\cS}$ appears exactly once as $f(v)$ and once as $g(v')$ for some $v,v'\in G'_{\cS}$, so the sum is zero.
 On the other hand, it clearly holds that $f(v) - g(v) = b_g - b_f$ for all
 $v$. Letting $n$ be the number of nodes in $G'_{\cS}$, we get the equation
 \[
	 0=n(b_g-b_f) \mbox{ in }\vF_p.
 \]
 Since $b_f\neq b_g$ by hypothesis, we must have $p \mid n$. This is impossible
 however, since $G'_{\cS}$ is not empty and consists only of nonsquares, so
 $1 \le n \le \frac{p-1}{2}$.
\end{proof}
The fact that the polynomials of Proposition \ref{prop:p3} don't have a linear
term is of crucial importance. Let us see why by giving an explicit example of a
semigroup of irreducible polynomials in $\vF_p[x]$ for which Proposition
\ref{prop:p3} does not apply (but $p\equiv 3 \mod 4$).
\begin{example}\label{ex:sharp_cond}
Let us fix $p=7$ and 
\[f=(x-1)^2-5=x^2+5x+3\in \vF_7[x]\]
\[g=(x-4)^2-5=x^2+6x+4\in \vF_7[x].\]
The set $\cS=\{f,g\}$ has distinguished set $D_\cS=\{-5\}$ and graph as in Figure \ref{fig:irred_semigroup_f7}.
\begin{figure}[h]
\centering
\begin{tikzpicture}
	\node (a) at (2,0) [draw, circle, double] {$-5$};
	\node (b) at (-2,0) [draw, circle]         {$3$};
    \node (c) at (0,-2) [draw, circle]         {$-1$};
	\draw [->] (a) to [out=500,in=30]  node[above] {$f$} (b);
	\draw [->] (a) to [out=225,in=30] node[above] {$g$} (c);
	\draw [->] (b) to [out=225,in=200] node[above] {$f$} (c);
	\draw [->] (b) to [loop left]     node[above] {$g$} (b);
	\draw [->] (c) to [loop right]     node[right] {$f$} (b);
	\draw [->] (c) to [loop above]     node[right] {$g$} (b);
\end{tikzpicture}
\caption{The nodes of $G_{\cS}$ reachable from $-5$.}
\label{fig:irred_semigroup_f7}
\end{figure}
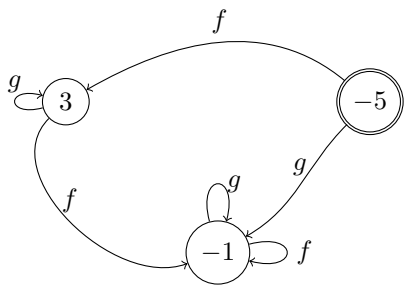
Since $5$ is not a square, and we only look at paths of positive length, the final claim follows by checking that $3$ and $-1$ are not squares modulo $7$.
\end{example}
\bibliographystyle{unsrt}
\bibliography{biblio}{}

\end{document}